\documentclass[preprint,10pt]{elsarticle}
\usepackage{float}  
 \usepackage{graphicx}
\usepackage{amssymb}
\usepackage{enumitem}    % ??????
 \usepackage{amsthm}
 \theoremstyle{plain}
 \newtheorem{theorem}{Theorem}[section]

 \newtheorem{lemma}{Lemma}[section]
 \newtheorem{proposition}{Proposition}[section]
 \newtheorem{corollary}{Corollary}[section]
\theoremstyle{definition}
 \newtheorem{definition}{Definition}[section]
 \newtheorem{example}{Example}[section]
 \newtheorem{remark}{Remark}
\theoremstyle{remark}
\usepackage[english]{babel}
\usepackage{amssymb,amstext,amsmath}
\usepackage{amsfonts,amsmath,amssymb,amscd}
\usepackage{amsfonts}
\usepackage{mathrsfs}
 \usepackage{lineno}
\journal{}

\begin{document}

\begin{frontmatter}
 \title{Laws of Large Numbers for Non-Independent Random Variables on Hyperspaces with respect to the Hausdorff Metric}

\author[label1]{Jinping Zhang}
\address[label1]{School of Mathematics and Physics, North China Electric Power University,
	Beijing, 102206, P.R.China}
\ead{zhangjinping@ncepu.edu.cn}
\author[label2]{Li Guan\tnoteref{label3}}
\tnotetext[label3]{Corresponding author:Li Guan}
\address[label2]{School of Mathematics, Statistics and Mechanics, Beijing University of Technology, 100 Pingleyuan, Chaoyang District, Beijing, 100124,
P.R.China}
\ead{guanli@bjut.edu.cn}

\begin{abstract}
	This paper investigates the limit behavior of the Minkowski sums for sequences of  set-valued random variables. When the underlying space is finite dimensional, by using the support function, we establish the weak and strong laws of large numbers for non-independent random variables in the hyperspace with respect to the Hausdorff metric $d_H$.

\end{abstract}


\end{frontmatter}

 %\linenumbers

%%%% Start %%%%%%
\section{Introduction}
\label{author_sec:1}
In the real world, since the uncertainty includes not only randomness but also imprecision, it has limitation to describe an event by single-valued random variables. The set-valued random variable is a suitable tool to characterize both randomness and imprecision. For example, the price of
a stock within one trading day may change a lot. The single-valued opening price or closing pricing is not enough to describe the uncertainty of market. It is more reasonable to consider the variation range of stock price, which can be described by a random interval, the special case of set-valued random variable. In the past few decades, the theory of set-valued random variables with a wide range of applications
has received a lot of attention. See for example \cite{Au65,CasVa,Frank,HiUm,LiOgV,Mol} and references therein. Especially, for the applications of set-valued theory to econometrics and finance, we would like to recommend the nice article \cite{Mol2014} and references therein.

It is well known that the limit theory plays an important role in classical probability theory, statistical inference and parameter estimation. The law of large numbers is an important limit theorem, which has been extensively studied and extended to set-valued cases, Choquet integrals, nonlinear expectations, and Sugeno integrals (\cite{Teran25}). For set-valued case, the first strong law of large numbers (in short by SLLNs) was given by Artstein
and Vitale in 1975 \cite{AV}, where the set-valued random variables are independent identically distributed and take values in the family of compact subsets of Euclidean space $\mathbb R^{d}$. After that, many other authors such as Gin$\acute{e}$,
Hahn and Zinn \cite{GHZ}, Hess \cite{He},  Hiai \cite{Hiai} studied  SLLNs under different
settings for convex set-valued random variables where the underlying
space is a separable Banach space.  By using The R\aa{}dstrom embedding theorem together with the SLLN in Banach spaces, Puri and Ralescu in 1983 \cite{PuRa83} concisely proved a SLLN for the sequence of independent and identically distributed compact convex set-valued random variables.  Artstein and Hansen \cite{AV1985}, Hiai \cite{Hiai} obtained the SLLNs for independent set-valued random variables in Banach space without convexity assumption by using different methods. Taylor and Inoue proved SLLNs
for independent (not necessary to be identically distributed)
case in Banach space in \cite{TaIna}. Ter$\acute{a}$n and Molchanov \cite{Teran06} studied the law of large numbers in a metric
space with a convex combination operation and applied the method in set-valued setting. Guan et al. \cite{Guan07} studied the SLLNs for weighted sums of set-valued random variables in Rademacher type p Banach space.  As the extension of set-valued case, Li and Ogura \cite{LiOg06} studied the SLLNs for independent (not necessary identically distributed) fuzzy set-valued variables in the sense of extended Hausdorff distance. Guan and Li \cite{Guan04}, Guan et al \cite{Guan08} obtained the strong law of large numbers for weighted sum of fuzzy set-valued variables.

The assumption of independence of random variable sequences is a little too strong for some cases. Considering weaker assumption, some researchers studied limit theorems for single-valued random variable sequences. There are dozens of papers studying weak and strong laws of large numbers for single-valued random variables which are not independent. For example, Taylor 1978 in \cite{Taylor78} defined uncorrelation of random variables and obtained laws of large numbers. Jan 2021 \cite{Jan} studied the Kolmogorov strong law of large numbers for uncorrelated real-valued random variables. Laws of large numbers also were studied under other weaker assumptions such as positive dependence, negative dependence \cite{Boz,Kuc,Taylor01}, dependence \cite{Pat}. Ko \cite{Ko} obtained the strong law of large numbers for linear multi-parameter stochastic processes generated by identically distributed and negatively associated random fields.

But for set-valued case, to our knowledge, for the law of large numbers, there is no result other than the condition of independence. It is also necessary and possible to study the limit behavior of set-valued random variables which are non-independent. Compared with the existing literature, the innovation of this paper is that the law of large numbers holds without the independence assumption of the sequence of random variables. Firstly, by using the support function, we define the uncorrelated set-valued random variables based on the notion of single-valued case in \cite{Taylor78}. Uncorrelation is weaker than independence.   Secondly, under the assumptions of uncorrelation and compactly uniform integrability, we shall prove the weak and strong laws of large numbers for the sequence of set-valued random variable in the sense of the Hausdorff metric $d_H$.

This paper is organized as follows. Section \ref{author_sec:2} is about some definitions and basic results of set-valued random variables. Section \ref{author_sec:3} contributes to the weak law of large numbers and examples.  Section \ref{author_sec:4} is on the strong law of large numbers. Section \ref{author_sec:5} gives the concluding remark.
%%%%%%%%%%%%%%%%%%%%%%%%%%%%%%%%%%%%%%%%%%%%%%%%%%%%%%%%%%

\section{Preliminaries}
\label{author_sec:2}
Throughout this paper, we assume that $(\Omega,{\mathcal{A}},\mu)$ is
a complete and non-atomic probability space. $(\frak X , \|\cdot\|)$ is a real
separable Banach space. Its dual space is denoted by $(\frak X^* , \|\cdot\|_{\frak X^*})$ . Let ${\mathcal P}(\frak X)$ (resp. ${\mathcal P}_k(\frak X)$) denote the family of all nonempty
closed (resp.compact) subsets of $\frak X$. And ${\mathcal P}_{kc}(\frak X)$ is the family of all
nonempty compact convex subsets of $\frak X$. Let $\mathbb R$ denote the family of all real numbers and $\mathbb N$ the set of all natural numbers.

Let $A$ and $B$ be two nonempty subsets of $\frak X$ and let
$\lambda\in\mathbb R$ be the set of all real numbers. The Minkowski sum
and scalar multiplication are defined by
$$A+B=\{a+b:a\in A, b\in B\}$$
$$\lambda A=\{\lambda a:a\in A\}$$

The Hausdorff metric on ${\mathcal P}(\frak X)$ is defined by
$$d_{H}(A,B)=\max\{\sup\limits_{a\in A}\inf\limits_{b\in B}\|a-b\|,\
\sup\limits_{b\in B}\inf\limits_{a\in A}\|a-b\|\}$$ for $A,\ B\in
{\mathcal P}(\frak X)$. For an $A$ in ${\mathcal P}(\frak X)$,  let $\|A\|_{{\mathcal P}}=d_{H}(\{0\},A)=\sup_{a\in A}\|a\|$.

The metric space $({\mathcal P}_k({\frak X}),d_H)$ is complete and separable. And
${\mathcal P}_{kc}({\frak X})$ is a closed subset of $( {\mathcal P}_k({\frak X}),d_H)$ (cf.
\cite{LiOgV}, Theorems 1.1.2 and 1.1.3).

For each $A\in {\mathcal P}({\frak X})$, the support function is defined by
$$\sigma(x^*,A)=\sup\limits_{a\in A}<x^*,a>,\ \ x^*\in {\frak X}^*.$$

Let $S^*$ be the unit
sphere in ${\frak X}^*$, and $C(S^*)$ be the set of all continuous
functions $f$ on $S^*$ with respect to the norm
$\|f\|_C=\sup_{x\in S^*}|f(x)|$. The mapping
$j_0:A\rightarrow \sigma(\cdot,A)$ can embed the space
$({\mathcal P}_{kc}({\frak X}),d_H)$ into a closed convex cone in $C(S^*)$
isometrically and isomorphically (cf. \cite{LiOgV}, Theorem 1.1.12
). By using the support function, we have the following equivalent definition of Hausdorff metric.
For $A,B\in {\mathcal P}_{kc}(\frak X)$, the Hausdorff metric between $A$ and $B$ is
\begin{equation}\label{metric}
	d_{H}(A,B)=\sup\{|\sigma(x^{*},A)-\sigma(x^{*},B)|:x^{*}\in S^{*}\}.
\end{equation}
Therefore, $\|A\|_{{\mathcal P}}=d_{H}(\{0\},A)=\sup_{a\in A}\|a\|=\sup_{x^{*}\in S^{*}}| \sigma(x^{*},A)|$.

A set-valued mapping $F:\Omega\rightarrow {\mathcal P}({\frak X})$ is
called {\em a set-valued random variable (or a random set, or a
	multifunction)} if, for each open subset $O$ of ${\frak X}$,
$F^{-1}(O)=\{\omega \in \Omega: F(\omega)\cap O \neq \emptyset \}\in
\mathcal A$. The family of all ${\mathcal P}({\frak X})$-valued random variables is denoted by $\mathcal M(\Omega; {\mathcal P}(\frak X))$.

A set-valued random variable $F$ is called {\em $L^{p}$-integrably bounded}($p
\geq 1$)
(cf. \cite{HiUm} or \cite{LiOgV}) if
$\int_{\Omega}\|F(\omega)\|_{\mathcal P}^{p}d\mu <\infty$.

Let $L^p(\Omega,\mathcal{A},\mu;{\mathcal P}_k(\frak X))$ (resp. $L^p(\Omega,\mathcal{A},\mu;{\mathcal P}_{kc}(\frak X))$) denote the space of all integrably bounded compact (resp. compact and convex) random variables, which is briefly denoted by $L^p(\Omega;{\mathcal P}_k(\frak X))$ (resp. $L^p(\Omega;P_{kc}(\frak X))$).
For $F, G\in
L^1(\Omega,\mathcal{A},\mu;{\mathcal P}_k(\frak X))$, $F=G$ if and only if
$F(\omega) =G(\omega)~ a.s.$  Regarding the
concepts and results of set-valued random variables, readers may
refer to nice books \cite{CasVa, LiOgV, Mol}.

For each set-valued random variable $F$, the expectation of $F$ is the Aumman integral,
denoted by $E[F]$, 
$$E[F]:=\{E[f]:f\in S_{F}\},$$
where $E[f]=\int_{\Omega}fd\mu$ is the usual Bochner integral in
$L^1[\Omega;\frak X]$ (the family of integrable $\frak X$-valued random
variables), and $S_F$ is the family of all integrable selections of $F$. $S_F:=\{f\in L^1(\Omega;\frak X) :f(\omega)\in F(\omega) \
a.s.\}$.

About support function, we list here some results that will be needed in later proofs. We refer the reader to page 421 in \cite{Mol}, page 7 and Theorem 2.1.12 in \cite{LiOgV}).
\begin{proposition}\label{prop}
	Take $A,B\in \mathcal P_{kc} (\frak X)$, $x_{1}^{*}, x^{*}_{2}\in\frak X^* $, $\lambda\geq 0$, then we have
	\begin{enumerate}[label=\upshape(\arabic*),align=left]
		\item  $\sigma(x^*,A+B)=\sigma(x^*,A)+\sigma(x^*,B)$;
		\item  $\sigma(x^*, \lambda A)=\lambda \sigma(x^*,A)$;
		\item  $\sigma(x_{1}^*+x_{2}^*, A)\leq \sigma(x_{1}^*,A)+ \sigma(x_{2}^*,A) $;
		\item $ |\sigma(x_{1}^*,A)- \sigma(x_{2}^*,A)|\leq \|x^*_1-x^*_2\|_{\frak X^*}\|A\|_{\mathcal P}$.
	\end{enumerate}
\end{proposition}
The formula $(4)$ in Proposition \eqref{prop} implies  that $\sigma(\cdot, A)$ is Lipschitz continuous with respect to $x^*\in\frak X^*$.

\begin{lemma}\label{lem:change}
	Let $F$ be a ${\mathcal P}(\frak X)$-valued random variable and $S_F\neq\emptyset$. Then for any $x^*\in{\frak X}^*$, we have
	$$
	\sigma(x^*, E[F])=E[\sigma(x^*, F)].
	$$
\end{lemma}

\begin{definition}\label{compactlyuniform} A sequence of random variables $\{F_{n}, n=1,2,\cdots\}\subset \mathcal M(\Omega;{\mathcal P}_k(\frak X))$ is called {\it compactly uniformly integrable in $L^{p}$ ($p\geq 1$)} if for each $\epsilon>0$, there exists a compact subset $\mathcal{K}_{\epsilon}\subset {\mathcal P}_{k}(\frak X)$ such that $E\left[\|F_{n}(\omega)\|^{p}_{\mathcal P} {I}_{\{F_n(\omega)\notin \mathcal{K}_{\epsilon}\}}\right]< \epsilon$,
	where $I$ is the indicator function. 
\end{definition}

$\mathbb R$-valued random variables $\xi$ and $ \eta$ are said to be {\em uncorrelated } if $cov (\xi,\eta)=0$ . In the following, we define the uncorrelated set-valued random variables.

\begin{definition}\label{def:uncorrelated}
	Let $F_{1},F_{2}$ be set-valued random variables.  $F_{1}$ and $F_{2}$ are said to be {\em uncorrelated} if $\sigma(x^*,F_1)$ and $\sigma(x^*,F_2)$ are
	uncorrelated $\mathbb R$-valued random variables  for any $x^*\in \frak X^*$.
	
	A sequence of set-valued random variables $\{F_1,F_2,\cdots\}$ is said to be  {\em pairwise uncorrelated} if  the sequence $\{\sigma(x^*,F_1), \sigma(x^*,F_2),\cdots\}$ is pairwise uncorrelated for any $x^*\in \frak X^*$. 
\end{definition}
`pairwise uncorrelated' will be referred to simply as `uncorrelated' hereafter.
%%%%%%%%%%%%%%%%%%%%%%%%%%%%%%%%%%%%%%%%%%%%%%%%%%%%%%%%%%%%%%%%%%%%
\section{Weak law of large numbers}
\label{author_sec:3}

In this section, firstly, we shall study the weak law of large numbers for interval-valued uncorrelated random variables.  Then we will extend the result to the case with underlying space $\frak X=\mathbb R^d$ ($1<d<\infty$). 

To judge the uncorrelation of two interval-valued random variables, it reduces to consider the uncorrelation of endpoints according to the following result.
\begin{theorem}\label{thm:interval}
	For interval-valued random variables $F=[f_-,f_+], G=[g_-,g_+]\in \mathcal M(\Omega; {\mathcal P}_k(\mathbb R))$,
	$F$ and $G$ are uncorrelated if and only if $f_-$ and $g_-$ are  uncorrelated, $f_+$ and $g_+$ are uncorrelated as well.
\end{theorem}
\begin{proof}
	For $\frak X=\mathbb R$, the unit sphere $S^*=\{1,-1\}$.
	
	If $F$ and $G$ are uncorrelated. That is, for $x^*\in \{1,-1\}$,  $\sigma(x^*,F)$ and $\sigma(x^*,G)$ are uncorrelated.
	
	When $x^*=1$, $\sigma(x^*,F)=f_+$ and 
	$\sigma(x^*, G)=g_+$. Then $f_+$ and $g_+$ are uncorrelated. When $x^*=-1$, similarly,
	$f_-$ and $g_-$ are uncorrelated.
	
	Conversely, assume $f_{+}$ and $g_{+}$ are uncorrelated, neither are $f_{-}$ and $g_{-}$. Then for $x^{*}=1$,
	$\sigma(x^*, F)$ and $\sigma(x^*,G)$ are uncorrelated. For any $x^*\in \mathbb R^*$, there exists an $a\in \mathbb R$, such that $x^*=a\times 1$. If $a\geq 0$
	$\sigma(x^*,F)=af_+$ and 	$\sigma(x^*,G)=ag_+$. If  $a< 0$,
	$\sigma(x^*,F)=af_-$ and 	$\sigma(x^*,G)=ag_-$.
	Then $\sigma(x^*,F)$ and $\sigma(x^*,G)$ are uncorrelated. By Definition \ref{def:uncorrelated}, $F, G$ are uncorrelated.
\end{proof}

\begin{example}
	Assume the real-valued vector $(\xi, \eta)$ are uniformly distributed in the elliptical disk
	$$\left\{(x,y)\in \mathbb R^{2}\mid\frac{(x-2)^2}{2^2}+\frac{(y-3)^2}{3^2}\leq 1\right\}.$$
	By simple calculation, we know that $\xi$ and $\eta$ are dependent and non-identically distributed, yet the correlation coefficient $\rho(\xi, \eta)=0$. Then by Theorem \ref{thm:interval}, the two interval-valued random variables $[0, \xi]$ and $[0, \eta]$ are uncorrelated while $ [0, \xi]$ and $[0,\eta]$ are not independent and follow different distributions.
	
In general, suppose non-negative real-valued random variables $\left \{f_{1}, \cdots, f_{n}, \cdots\right\}$ are uncorrelated, non-independent and non-identically distributed. Then the interval-valued random variables  $\left \{[0, f_{1}], \cdots, [0, f_{n}], \cdots\right\}$ satisfy the same properties.
\end{example}

\vskip10pt

Now we give a weak law of large numbers for uncorrelated ${\mathcal P}_{kc}(\mathbb R)$-valued random variables.

\begin{theorem} \label{thm:2}
	Let $\{V_{n}:n\in \mathbb N\}$ be a sequence of uncorrelated ${\mathcal P}_{kc}(\mathbb R)$-valued random variables such that for each n,
	$Var(\sigma(x^*,V_{k}))$ exists and for any $x^*\in \mathbb R^*(=\mathbb R)$,
	$$\frac{1}{n^2}\sum_{k=1}^{n}Var(\sigma(x^*,V_{k}))\longrightarrow0\ as \  n\rightarrow\infty.$$
	Then
	\begin{equation}\label{eqn:result00}
		P\Big\{d_{H}\Big(\frac{1}{n}\sum\limits_{k=1}^n V_{k},\frac{1}{n}\sum\limits_{k=1}^n E[V_{k}]\Big)>\varepsilon\Big\}\longrightarrow0 \ as \  n\rightarrow\infty.
	\end{equation}
\end{theorem}
\begin{proof}
	For any $\varepsilon>0$. By the Markov inequality and the equivalent definition of Hausdorff metric, we have
	
	\begin{equation}\label{eqn:equivalent}
		\begin{split}
			&P\Big\{d_{H}\Big(\frac{1}{n}\sum\limits_{k=1}^n V_{k},\frac{1}{n}\sum\limits_{k=1}^n E[V_{k}]\Big)>\varepsilon\Big\}\\
			&\leq\frac{1}{(\varepsilon n)^{2}}E\Big[d_{H}\Big(\sum\limits_{k=1}^n V_{k},\sum\limits_{k=1}^n E[V_{k}]\Big)\Big]^{2}\\
			&=\frac{1}{(\varepsilon n)^{2}}E\Big[\sup_{x^{*}\in S^{*}}|\sigma(x^{*}, \sum\limits_{k=1}^nV_{k})-\sigma(x^{*}, \sum\limits_{k=1}^nE[V_k])|\Big]^2\\
			&=\frac{1}{(\varepsilon n)^{2}}E\Big[\sup_{x^{*}\in S^{*}}|\sigma(x^{*}, \sum\limits_{k=1}^nV_{k})-\sigma(x^{*}, \sum\limits_{k=1}^nE[V_k])|^2\Big]
		\end{split}
	\end{equation}
	
	For the real space $\mathbb R$, $S^{*}=\{1,-1\}$. Denote $\sigma(1,V_{k})=y_{k}$, $\sigma(-1,V_{k})=z_{k}$. Then by Definition \ref{def:uncorrelated}, both $\{y_{k}:k\geq 1\}$ and  $\{z_{k}:k\geq 1\}$ are uncorrelated real-valued random variable sequences.
	
	\noindent Therefore
	\begin{equation}\label{eqn:limit}
		\begin{split}
			&\frac{1}{(\varepsilon n)^{2}}E\Big[\sup_{x^{*}\in S^{*}}|\sigma(x^{*},\sum\limits_{k=1}^n
			V_{k})-\sigma(x^{*},\sum\limits_{k=1}^n E[V_{k}])|^{2}\Big]\\
			&\leq\frac{1}{(\varepsilon n)^{2}}E\Big[|\sum\limits_{k=1}^n\sigma(1, V_{k})-\sum\limits_{k=1}^n\sigma(1,
			E[V_{k}])|^{2}\\
			&\hspace{0.5cm}+|\sum\limits_{k=1}^n\sigma(-1, V_{k})-\sum\limits_{k=1}^n\sigma(-1, E[V_{k}])|^{2}\Big]\\
			&=\frac{1}{(\varepsilon n)^{2}}E\Big[(\sum\limits_{k=1}^ny_{k}-\sum\limits_{k=1}^nE[y_{k}])^{2}
			+(\sum\limits_{k=1}^nz_{k}-\sum\limits_{k=1}^nE[z_{k}])^{2}\Big]\\
			&=\frac{1}{(\varepsilon n)^{2}}E\Big[\sum\limits_{k=1}^n(y_{k}-E[y_{k}])^{2}+\sum\limits_{k\neq
				l}(y_{k}-E[y_{k}])(y_{l}-E[y_{l}])\\
			&\hspace{0.5cm}+\sum\limits_{k=1}^n(z_{k}-E[z_{k}])^{2}+\sum\limits_{k\neq l}(z_{k}-E[z_k])(z_{l}-E[z_l])\Big]\\
			&=\frac{1}{(\varepsilon
				n)^{2}}E\Big[\sum\limits_{k=1}^n(y_{k}-E[y_k])^{2}+\sum\limits_{k=1}^n(z_{k}-E[z_k])^{2}\Big]\\
			&=\frac{1}{(\varepsilon n)^{2}}\Big[\sum\limits_{k=1}^nVar(y_{k})+\sum\limits_{k=1}^nVar(z_{k})\Big]\\
			&\longrightarrow0\ \ \ {\mbox as} \ \  n\rightarrow\infty,
		\end{split}
	\end{equation}
	which together with \eqref{eqn:equivalent} yields the result \eqref{eqn:result00}.
\end{proof}

\vskip10pt
The following is an example of weak law of large numbers, in which the sequence is uncorrelated but is non-independent and has no identical distribution. That means the condition is weaker than the existing results of laws of large numbers such as in \cite{AV,Hiai,TaIna} etc.

\begin{example}\label{ex:3}
	Let $Y$ be a random variable following a Bernoulli distribution $B(1,\frac{1}{2})$, and let $\{X_n​,n=1,2, \cdots,\}$ be independent and identically distributed random variables with uniform distribution $ U(−1,1)$, which are independent of Y. Define $Z_n​=\frac{n}{n+1}X_n Y$, and let the interval-valued random variable $V_n​:=[Z_n​,Z_n​+1]$. Then the sequence $\{V_n​,n=1,2,\cdots,\}$ is uncorrelated, non-independent, non-identically distributed, and obeys the weak law of large numbers.
\end{example}
\begin{proof}
	It is easy to see that
	 $$E[Z_n]=\frac{n}{n+1}E[X_nY]=\frac{n}{n+1}E[X_n]E[Y]=0,$$
	 $$E[Z_nZ_m]=E[X_nX_mY^2]=\frac{n}{n+1}\frac{m}{m+1}E[X_n]E[X_m]E[Y^2]=0 \ for \ m\ne n,$$
	 which means that $X_m$ and $X_n$ are uncorrelated.
	 In addition, for $m\ne n$
	 $$P(X_mY\le 0.5, \ Y_nY\le 0.5)=\frac{25}{32}\ne P(X_nY\le 0.5)P(X_mY\le 0.5)=(\frac{7}{8})^2,$$
	 then $Z_n$ and $Z_m$ are non-independent.
	 
	 $$Var(Z_n+1)=Var(Z_n)=(\frac{n}{n+1})^2E[(X_nY)^2]=(\frac{n}{n+1})^2E[X_n^2]E[Y^2]=\frac{1}{6}(\frac{n}{n+1})^2.$$
	 Therefore, 
	 	$$\frac{1}{n^2}\sum_{i=1}^{n}Var(\sigma(x^*,V_{i}))\le \frac{1}{n^2}\frac{n}{3}(\frac{n}{n+1})^2\rightarrow 0 \ \ as \ n\rightarrow \infty.$$
	 	By Theorem \ref{thm:2},  $\{V_n​,n=1,2,\cdots,\}$ obeys the weak law of large numbers. 
\end{proof}
To make the conclusion of Example \ref{ex:3} more intuitive, we visualize the results using R software (version 4.3.1). Figure \ref{fig1} is the simulation result. The solid blue vertical lines represent the intervals corresponding to the sample mean $\frac{1}{n}​\sum_{i=1}^{n}​V_i​$ for sample sizes $n$ from $1$ to $300$. For clarity, only $50$ of these sample means are displayed. The two red dashed lines represent the left and right endpoints of the arithmetic average of the expectations, which is the interval $[0,1]$. The simulation results show that as the sample size increases, the deviation between the left and right endpoints of the sample mean and the expected mean becomes smaller.
\begin{figure}[H]
	\centering
	\includegraphics[width=0.75\textwidth]{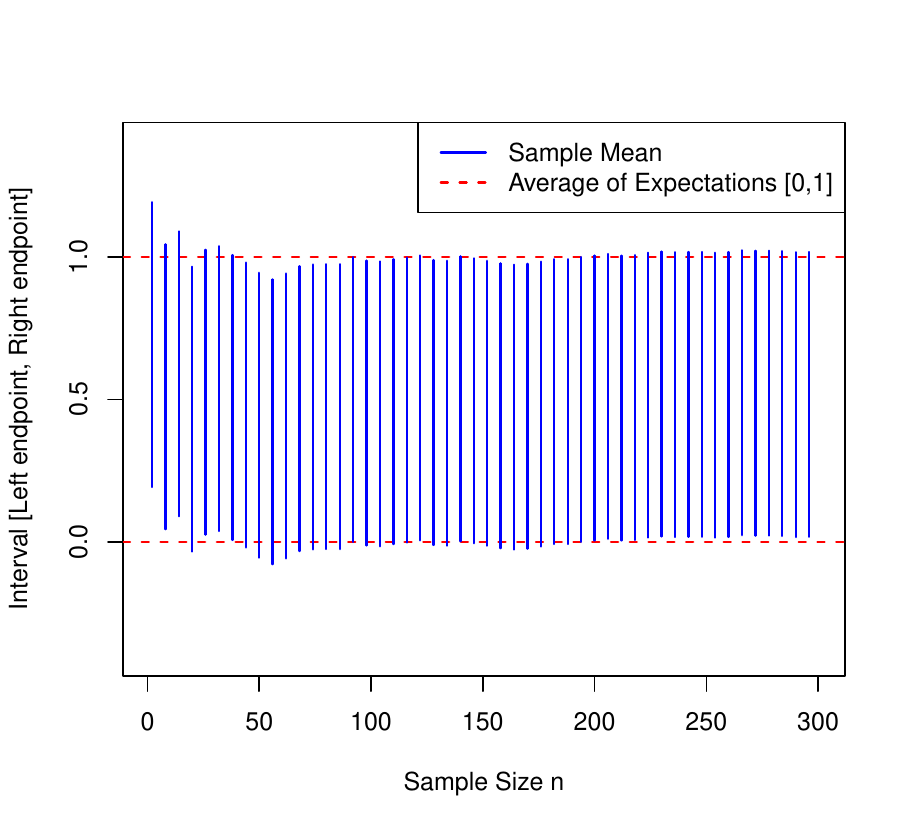}
	\caption{Simulation: Sample Mean vs Theoretical Interval [0,1]}
	\label{fig1}
\end{figure}

Here is another example.
\begin{example}\label{ex2}
	Let the real-valued random vector $(X_1, X_2,\cdots,X_n)$ be uniformly distributed in the following $n$-dimensional ellipsoid
	$$\frac{x_1^2}{a_1^2}+\frac{x_2^2}{a_2^2}+\cdots+\frac{x_n^2}{a_n^2}\leq 1.$$
	We assume that $a_{i}>0$ for each $i$ and the equality $a_1=a_2=\cdots=a_n$ does not hold. The joint density function is
	$$f(x_1,x_2,\cdots,x_n)=\left\{\begin{array}{cc}
		\frac{\Gamma(\frac{n}{2}+1)}{\pi^{\frac{n}{2}}\prod\limits_{i=1}^n a_i},\ \ \ x_i\in(-a_i,a_i), \ i=1,2,\cdots,n;\\
		0,\ \ \ \ \hspace{1cm}otherwise.
	\end{array}
	\right.
	$$
		Define 
		$$(Y_1,Y_2,\cdots,Y_n)=(X_1+a_1,X_2+a_2,\cdots,X_n+a_n)$$ and $V_1=[X_1,Y_1],\cdots,V_n=[X_n,Y_n]$.
		
	Then for sufficiently large $n$, the arithmetic average of the random variables $V_1​,…,V_n​$, i.e.,$\frac{1}{n} \sum_{i=1}^{n}​V_i$​, can be approximated by the arithmetic average of their expectations $\frac{1}{n} \sum_{i=1}^{n}​E[V_i]$​.
	\end{example}
	\begin{proof}
	It is not difficult to obtain that $E[X_i]=0 (i=1,\cdots,n)\}$ and
	$cov[X_i, X_j]=0 \ for \ i\neq j.$
	Therefore the sequence $(X_1, X_2, \cdots,X_n)$ is uncorrelated with different distributions.
	For any pair $(i, j)$, we know that $X_i$ and $X_j$ are not independent. Then the sequence $X_1,X_2,\cdots,X_n$ are not independent.
	
	For each $i$, we have $Var(X_i)=E(X_i^2)=\frac{a_i^2}{n+2}$.

	Random variables $(Y_1,Y_2,\cdots,Y_n)$ are non-negative and uncorrelated with $E[Y_{i}]=a_{i}$ for each $i$. $(Y_1,Y_2,\cdots,Y_n)$ are neither independent nor  identical distributed.
	$(Y_1,Y_2,\cdots,Y_n)$ are  uniformly distributed in the ellipsoid
	$$\frac{(y_1-a_1)^2}{a_1^2}+\frac{(y_2-a_2)^2}{a_2^2}+\cdots+\frac{(y_n-a_n)^2}{a_n^2}\leq 1.$$
	Clearly, $Var(Y_i)=Var(X_i)=\frac{a_i^2}{n+2}$.
	
 We have $E[V_{i}]=[0,a_{i}]$ for each $i$. By Theorem \ref{thm:interval}, the interval-valued random sequence $V_1, \cdots,V_n$  are neither independent nor
	identical distributed but uncorrelated.
	By simple calculation, we have
	$$
	\frac{1}{n^2}\sum\limits_{i=1}^n Var(Y_i)=\frac{1}{n^2}\sum\limits_{i=1}^n \frac{a_i^2}{n+2}
	=\frac{1}{n^2(n+2)}\sum\limits_{i=1}^n a_i^2.
	$$
	
	For fixed $n$, we can confine $a_{1},\cdots, a_{n}$ such that $a_{i}\leq \sqrt{n}$ for each $i$. Then we obtain
	$$\frac{1}{n^2(n+2)}\sum\limits_{i=1}^n a_i^2\leq \frac{1}{n+2}
	\rightarrow 0\ \ \ \ \ \ as\ \ n\rightarrow\infty.
	$$
	Moreover, for any $\epsilon>0$
	\begin{equation*}\label{eqn:result1}
		P\Big\{d_{H}\Big(\frac{1}{n}\sum\limits_{i=1}^n V_{i},\frac{1}{n}\sum\limits_{i=1}^n [0,a_{i}]\Big)>\varepsilon\Big\}\leq \frac{2 \sum\limits_{i=1}^n Var(Y_i)}{(\epsilon n)^{2}}\longrightarrow0 \ \ as \ \  n\rightarrow\infty,
	\end{equation*}
	which implied the result.
	
\end{proof}
To provide a more intuitive illustration of the results from Example \ref{ex2}, we set  $a_{i}=1+\frac{i}{n}$ ($i=1,\dots,n$) and use R software to generate random numbers from the corresponding distribution for visualization.
 
Figure \ref{fig2} presents the random intervals $\{V_1,\dots,V_{100}\}$ when $\{X_{1},\dots,X_{100}\}$ are distributed on an $100$-dimensional ellipsoid.  The blue vertical lines represent the 100 sample intervals. The two solid red horizontal lines denote the left and right endpoints of the sample mean of these intervals. The two dashed green horizontal lines indicate the left and right endpoints of the mean of the expectations, respectively.

% ---------- Figure 2 ----------
\begin{figure}[H]
	\centering
	\includegraphics[width=0.65\textwidth]{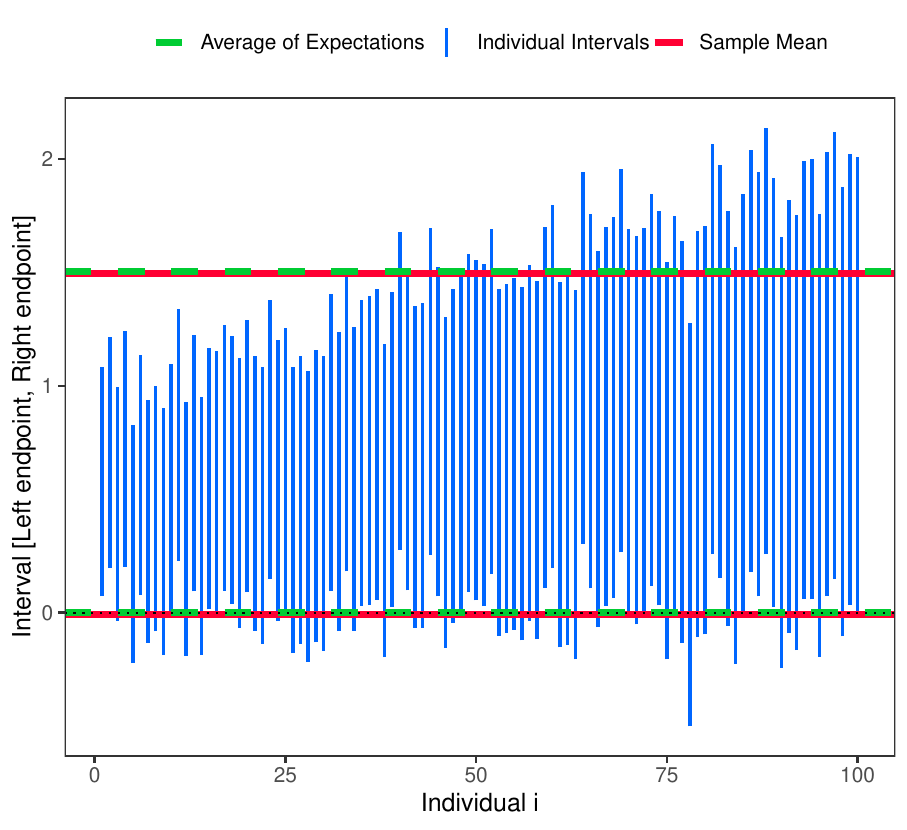}
	\caption{Random intervals $\{V_1,\dots,V_{100}\}$ generated from the uniform distribution on a $100$-dimensional ellipsoid.}
	\label{fig2}
\end{figure}

Figure \ref{fig3} illustrates the approximate relationship between the sample mean $\frac{1}{n}\sum_{i=1}^{n}V_i$ and the average of expectations $\frac{1}{n}\sum_{i=1}^{n}[0,a_i]$, where $V_i=[X_i, X_i+a_i]$, $a_i=1+\frac{i}{n}$ $ i=1,\cdots ,n$. And random variables $X_1,\dots,X_n$ follow the uniform distribution on an $n$-dimensional ellipsoid, where $n$ ranges from $5$ to $500$. It can be seen that the interval corresponding to the sample mean (red) and the interval corresponding to the average of expectations (blue) become closer as $n$ increases.

% ---------- Figure 3 ----------
\begin{figure}[H]
	\centering
	\includegraphics[width=0.65\textwidth]{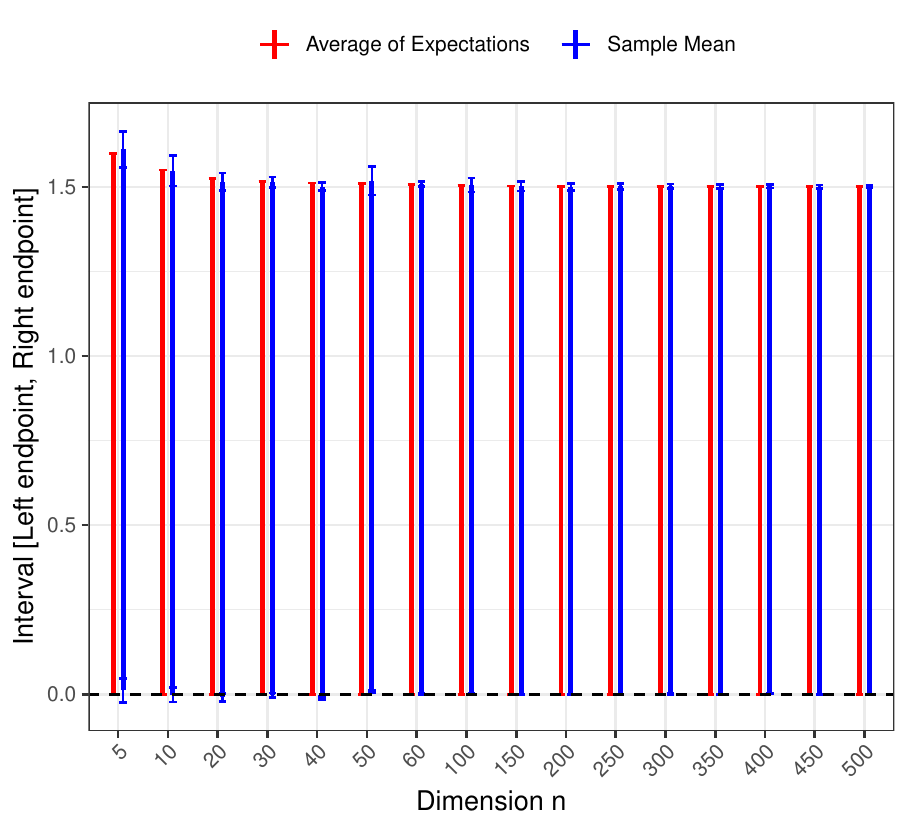}
	\caption{Approximate relationship between the sample mean and the expectation of the sample mean.}
	\label{fig3}
\end{figure}

For the underlying space  $\frak X=\mathbb R$, thanks to the unit sphere of $\frak X^*$ containing only two elements, the proof of WLLN  is relatively simpler, and the required conditions also are weaker. If the dimension of underlying space $\frak X$ is greater than $1$, even for finite-dimensional Euclidean space $\mathbb R^d (d>1)$, the unit sphere $S^*$ of $\frak X^*$ contains infinitely many elements. This makes the situation much more complex than $\mathbb R$ and it is easy to make mistakes if one is not careful. Here we only deal with the case of underlying space being of finite dimension.  For the sake of convenience in proving WLLN and SLLN for set-valued random variables, we first give the following Lemmas.
\begin{lemma}\label{compactuniform}
	Assume the dimension $dim\frak X<\infty$. Let $\{V_{n}:n\in \mathbb N\}$ be a sequence of ${\mathcal P}_{kc}(\frak X)$-valued  random variables, which are  $L^1$-integrably bounded. Then  for any $\varepsilon>0$, there exists finitely many points $\{y_1^*,\cdots,y_m^*\}\subset S^*$, such that for any $x^*\in S^*$, the following inequality holds:
	\begin{equation}\label{keyinequality}
		\begin{split}
			&\sup\limits_{x^*\in S^*}\Big|\sigma(x^*,\sum\limits_{k=1}^nV_k)-\sigma(x^*,\sum\limits_{k=1}^nE[V_k])\Big|\\
			&\leq \max\limits_{1\leq i\leq m}\Big\{\Big|\sigma(y_{i}^*,\sum\limits_{k=1}^nV_k)-\sigma(y_{i}^*,\sum\limits_{k=1}^nE[V_k])\Big|+\varepsilon
			\Big( \Big\|\sum\limits_{k=1}^nV_k\Big\|_{\mathcal P}+ \Big\|\sum\limits_{k=1}^nE[V_k]\Big\|_{\mathcal P}\Big)\Big\} .
		\end{split}
	\end{equation}
\end{lemma}
\begin{proof}
	For finitely dimensional space $\frak X$, the unit sphere $S^*$ of $\frak X^*$ is a compact set. Then for any $\varepsilon>0$, there exists a finite open cover $\{N(y_{i}^*; \varepsilon): i=1,\cdots,m\}$ such that $S^*\subset \bigcup\limits_{i=1}^{m}N(y_i^*;\varepsilon)$, where $y_i^*\in S^*$ and $N(y_i^*; \varepsilon)=\{y^*\in \frak X^*\mid \|y^*-y_i^*\|_{\frak X^*}<\varepsilon\}$. For any $x_i^* \in N(y_{i}^*;\varepsilon)$, by Proposition \ref{prop}, we have
	\begin{equation}\label{keyinequality1}
		\begin{split}
			&\Big|\sigma(x_i^*,\sum\limits_{k=1}^nV_k)-\sigma(x_i^*,\sum\limits_{k=1}^nE[V_k])-\sigma(y_{i}^*,
			\sum\limits_{k=1}^nV_k)+\sigma(y_{i}^*,\sum\limits_{k=1}^nE[V_k])\Big|\\
			&\leq \Big|\sigma(x_i^*,\sum\limits_{k=1}^nV_k)-\sigma(y_i^*,\sum\limits_{k=1}^nV_k)\Big|+\Big|\sigma(x_{i}^*,
			\sum\limits_{k=1}^nE[V_k])-\sigma(y_{i}^*,\sum\limits_{k=1}^nE[V_k])\Big|\\
			&\leq \|x_i^*-y_{i}^*\|_{\frak X^*}\Big\|\sum\limits_{k=1}^nV_k\Big\|_{\mathcal P}+\|x_i^*-y_{i}^*\|_{\frak X^*}\Big\|\sum\limits_{k=1}^nE[V_k]\Big\|_{\mathcal P}\\
			&=\|x_i^*-y_{i}^*\|_{\frak X^*}\Big( \Big\|\sum\limits_{k=1}^nV_k\Big\|_{\mathcal P}+ \Big\|\sum\limits_{k=1}^nE[V_k]\Big\|_{\mathcal P}\Big)\\
			&\leq  \varepsilon  \Big( \Big\|\sum\limits_{k=1}^nV_k\Big\|_{\mathcal P}+ \Big\|\sum\limits_{k=1}^nE[V_k]\Big\|_{\mathcal P}\Big).
		\end{split}
	\end{equation}
	Then by the triangle inequality and continuity of the support function $\sigma$ with respect to $x^*$ in the compact set $S^*$, it holds that
	\begin{equation}\label{keyinequality2}
		\begin{split}
			&\sup\limits_{x^*\in S^*}\Big|\sigma(x^*,\sum\limits_{k=1}^nV_k)-\sigma(x^*,\sum\limits_{k=1}^nE[V_k])\Big|\\
			&=\max\limits_{1\leq i\leq m}\Big\{\sup\limits_{x_i^*\in N(y^*;\varepsilon)}\Big|\sigma(x_i^*,\sum\limits_{k=1}^nV_k)-\sigma(x_i^*,\sum\limits_{k=1}^nE[V_k])\Big|\Big\}\\
			&\leq \max\limits_{1\leq i\leq m}\Big\{\Big|\sigma(y_{i}^*,\sum\limits_{k=1}^nV_k)-\sigma(y_{i}^*,\sum\limits_{k=1}^nE[V_k])\Big|\\
			& \hspace{0.5cm} +\sup\limits_{x^*_i\in N(y_i^*;\varepsilon)}\|x_i^*-y_{i}^*\|_{\frak X^*}\Big( \Big\|\sum\limits_{k=1}^nV_k\Big\|_{\mathcal P}+ \Big\|\sum\limits_{k=1}^nE[V_k]\Big\|_{\mathcal P}\Big)\Big\}\\
			&\leq \max\limits_{1\leq i\leq m}\Big\{\Big|\sigma(y_{i}^*,\sum\limits_{k=1}^nV_k)-\sigma(y_{i}^*,\sum\limits_{k=1}^nE[V_k])\Big|+\varepsilon
			\Big( \Big\|\sum\limits_{k=1}^nV_k\Big\|_{\mathcal P}+ \Big\|\sum\limits_{k=1}^nE[V_k]\Big\|_{\mathcal P}\Big)\Big\} 
			.
		\end{split}
	\end{equation}
\end{proof}
\begin{lemma}\label{expectsupinequality}
	Assume the dimension $dim\frak X<\infty$. Let $\{V_{n}:n\in \mathbb N\}$ be a sequence of ${\mathcal P}_{kc}(\frak X)$-valued  random variables, which are  uncorrelated and compactly uniformly integrable in $L^2$. Then  for any $\varepsilon>0$, there exists finitely many points $\{y_1^*,\cdots,y_m^*\}\subset S^*$, such that for any $x^*\in S^*$, the following inequality holds:
	\begin{equation}\label{expectsup}
		\begin{split}
			&\frac{1}{\varepsilon^2n^2}E\Big[\sup\limits_{x^*\in S^*}\Big|\sigma(x^*,\sum\limits_{k=1}^nV_k)-\sigma(x^*,\sum\limits_{k=1}^nE[V_k])
			\Big|^2\Big]\\
			&\leq \frac{1}{\varepsilon^2}\sum\limits_{i=1}^{m}\frac{\sum\limits_{k=1}^n Var(\sigma(y_{i}^*,V_k))}{n^2}
			+4\sqrt{\sum\limits_{i=1}^{m}\frac{\sum\limits_{k=1}^n Var(\sigma(y_{i}^*,V_k))}{n^2}}+4\varepsilon^2.
		\end{split}
	\end{equation}
\end{lemma}
\begin{proof}
	According to the Definition \ref{compactlyuniform}, for any $\varepsilon>0$, there exists a compact subset $\mathcal K_{\varepsilon}\subset \mathcal P_{k}(\frak X)$, such that
	$$E\left[\|V_n\|^2_{\mathcal P}I_{\{V_n\notin \mathcal K_{\varepsilon}\}}\right]<\varepsilon \ \ 	for \ \ all \ \ V_{n}.$$
	Then we have 
	\begin{equation}\label{squareintegral}
		E\left[\|V_n\|^2_{\mathcal P}\right]=E\left[\|V_n\|^2_{\mathcal P}I_{\{V_n\notin \mathcal K_{\varepsilon}\}}\right]+E\left[\|V_n\|^2_{\mathcal P}I_{\{V_n\in \mathcal K_{\varepsilon}\}}\right]
		\leq \varepsilon+\sup\limits_{A\in \mathcal K_{\varepsilon}}\|A\|^2_{\mathcal P}\triangleq M,
	\end{equation}
	where $M$ depends on $\varepsilon$.
	
	By \eqref{squareintegral}, we have $E\left[\|V_n\|_{\mathcal P}\right]<\sqrt{M}$ and
	\begin{equation}\label{bounded}
		\begin{split}
			&E\Big[ \Big\|\sum\limits_{k=1}^nV_k\Big\|_{\mathcal P}+ \Big\|\sum\limits_{k=1}^nE[V_k]\Big\|_{\mathcal P}\Big]^2\\
			&\leq E\Big[ \sum\limits_{k=1}^n(\|V_k\|_{\mathcal P}+ E[\|V_k\|_{\mathcal P}])\Big]^2\\
			&=\sum\limits_{i,j=1}^{n} E\Big[\left(\|V_i\|_{\mathcal P}+ E[\|V_i\|_{\mathcal P}]\right)(\|V_j\|_{\mathcal P}+E[\|V_j\|_{\mathcal P}])\Big]\\
			&\leq \sum\limits_{i,j=1}^n E\left[\|V_j\|_{\mathcal P}\|V_j\|_{\mathcal P}\right]+3\sum\limits_{i=1}^{n}E\left[\|V_i\|_\mathcal P\right]\sum\limits_{j=1}^{n}E\left[\|V_j\|_\mathcal P\right]\\
			&\leq 4 n^2M. \ (\ by \ Schwarz\  inequality\  and \ \eqref{squareintegral})
		\end{split}
	\end{equation}
	Applying Lemma \ref{compactuniform}, with $\varepsilon$	replaced by $\frac{\varepsilon^2}{\sqrt{M}}$, it follows that there exists $\{y^*_1,\cdots,y^*_m\}\subset S^*$, such that
	\begin{equation}\label{keyinequality3}
		\begin{split}
			&\sup\limits_{x^*\in S^*}\Big|\sigma(x^*,\sum\limits_{k=1}^nV_k)-\sigma(x^*,\sum\limits_{k=1}^nE[V_k])\Big|\\
			&\leq \max\limits_{1\leq i\leq m}\Big\{\Big|\sigma(y_{i}^*,\sum\limits_{k=1}^nV_k)-\sigma(y_{i}^*,\sum\limits_{k=1}^nE[V_k])\Big|+\frac{\varepsilon^2}{\sqrt{M}}
			\Big( \Big\|\sum\limits_{k=1}^nV_k\Big\|_{\mathcal P}+ \Big\|\sum\limits_{k=1}^nE[V_k]\Big\|_{\mathcal P}\Big)\Big\} 
			.
		\end{split}
	\end{equation}
	By the above inequality \eqref{keyinequality3}, it holds that
	\begin{eqnarray*}
		&&\frac{1}{\varepsilon^2n^2}E\Big[\sup\limits_{x^*\in S^*}\Big|\sigma(x_i^*,\sum\limits_{k=1}^nV_k)-\sigma(x_i^*,\sum\limits_{k=1}^nE[V_k])
		\Big|^2\Big]\\
		&&\leq \frac{1}{\varepsilon^2n^2}E\Big[\Big(\max\limits_{1\leq i\leq
			m}\Big|\sigma(y_{i}^*,\sum\limits_{k=1}^nV_k)-\sigma(y_{i}^*,\sum\limits_{k=1}^nE[V_k])\Big|\\
		&& \hspace{0.5cm}
		+\frac{\varepsilon^2}{\sqrt{M}}
		\Big( \Big\|\sum\limits_{k=1}^nV_k\Big\|_{\mathcal P}+ \Big\|\sum\limits_{k=1}^nE[V_k]\Big\|_{\mathcal P}\Big)\Big)^2\Big] \\
		&&=\frac{1}{\varepsilon^2n^2}E\Big[\max\limits_{1\leq i\leq
			m}\Big\{\Big|\sigma(y_{i}^*,\sum\limits_{k=1}^nV_k)-\sigma(y_{i}^*,\sum\limits_{k=1}^nE[V_k])\Big|^2 \\
		&&\hspace{0.5cm}
		+\frac{2\varepsilon^2}{\sqrt{M}}\Big|\sigma(y_{i}^*,\sum\limits_{k=1}^nV_k)-\sigma(y_{i}^*,\sum\limits_{k=1}^nE[V_k])\Big|
		\Big( \Big\|\sum\limits_{k=1}^nV_k\Big\|_{\mathcal P}+ \Big\|\sum\limits_{k=1}^nE[V_k]\Big\|_{\mathcal P}\Big)
		\Big\}\\
		&&\hspace{0.5cm}+\frac{\varepsilon^4}{M}\Big( \Big\|\sum\limits_{k=1}^nV_k\Big\|_{\mathcal P}+
		\Big\|\sum\limits_{k=1}^nE[V_k]\Big\|_{\mathcal P}\Big)^2
		\Big]\triangleq K.
	\end{eqnarray*}
	For any $x^*\in S^*$, by Lemma \ref{lem:change} and  Definition \ref{def:uncorrelated}, for $k\neq l$, we have
	\begin{equation}\label{eqn:2}
		\begin{split}
			&E\Big\{\Big[\sigma(x^{*},V_{k})-\sigma(x^{*},E[V_{k}])\Big]
			\Big[\sigma(x^{*},V_{l})-\sigma(x^{*},E[V_{l}])\Big]\Big\}\\
			&=E\Big\{\Big[\sigma(x^{*},V_{k})-E[\sigma(x^{*},V_{k})]\Big]
			\Big[\sigma(x^{*},V_{l})-E[\sigma(x^{*},V_{l})]\Big]\Big\}=0.
		\end{split}
	\end{equation}
	By Proposition \ref{prop}, Lemma \ref{lem:change}, Schwarz inequality and \eqref{eqn:2}, we have
	\begin{eqnarray*}
		&&K\leq \frac{1}{\varepsilon^2n^2}E\Big[\max\limits_{1\leq i\leq
			m}\Big|\sigma(y_{i}^*,\sum\limits_{k=1}^nV_k)-\sigma(y_{i}^*,\sum\limits_{k=1}^nE[V_k])\Big|^2\Big]\\
		&&\hspace{0.5cm}+\frac{2}{\sqrt{M}n^2}\sqrt{E\Big[\max\limits_{1\leq i\leq m}
			\Big|\sigma(y_{i}^*,\sum\limits_{k=1}^nV_k)-\sigma(y_{i}^*,\sum\limits_{k=1}^nE[V_k])\Big|^2\Big]}
		\sqrt{
			E\Big[\Big( \Big\|\sum\limits_{k=1}^nV_k\Big\|_{\mathcal P}+
			\Big\|\sum\limits_{k=1}^nE[V_k]\Big\|_{\mathcal P}\Big)^2\Big ]}\\
		&&\hspace{0.5cm}+ \frac{\varepsilon^2}{Mn^2}E\Big( \Big\|\sum\limits_{k=1}^nV_k\Big\|_{\mathcal P}+
		\Big\|\sum\limits_{k=1}^nE[V_k]\Big\|_{\mathcal P}\Big)^2
		\\
		&&\leq\frac{1}{\varepsilon^2n^2}\sum\limits_{i=1}^m \Big[\sum\limits_{k=1}^n E[\sigma(y_{i}^*,V_k)-\sigma(y_{i}^*,
		E[V_k])]^2\\
		&& \hspace{0.5cm}+\sum\limits_{k\neq j}E[\sigma(y_{i}^*,V_k)-\sigma(y_{i}^*, E[V_k])][\sigma(y_{i}^*,V_j)-\sigma(y_{i}^*, E[V_j])]\Big]\\
		&&\hspace{0.5cm}+\frac{2}{\sqrt{M}n^2}\sqrt{E\Big[\max\limits_{1\leq i\leq m}
			\Big|\sigma(y_{i}^*,\sum\limits_{k=1}^nV_k)-\sigma(y_{i}^*,\sum\limits_{k=1}^nE[V_k])\Big|^2\Big]}
		\sqrt{4Mn^2}+ 4\varepsilon^2 \\
		&&=\frac{1}{\varepsilon^2n^2}\sum\limits_{i=1}^m \sum\limits_{k=1}^n Var(\sigma(y_{i}^*,V_k))
		+\frac{4}{n}\sqrt{\sum\limits_{i=1}^m\sum\limits_{k=1}^n Var(\sigma(y_{i}^*,V_k))}+4\varepsilon^2\\
		&&\leq \frac{1}{\varepsilon^2}\sum\limits_{i=1}^{m}\frac{\sum\limits_{k=1}^n Var(\sigma(y_{i}^*,V_k))}{n^2}
		+4\sqrt{\sum\limits_{i=1}^{m}\frac{\sum\limits_{k=1}^n Var(\sigma(y_{i}^*,V_k))}{n^2}}+4\varepsilon^2.
	\end{eqnarray*}	 
\end{proof}

Having established these two lemmas, we are now able to prove the weak law of large numbers.
\begin{theorem}\label{thm:result2}
	Assume the dimension $dim\frak X<\infty$.  Let $\{V_{n}:n\in \mathbb N\}$ be a sequence of ${\mathcal P}_{kc}(\frak X)$-valued random variables, which are  uncorrelated and compactly uniformly integrable in $L^2$, and for any $x^*\in S^*$
	\begin{equation}\label{eqn:uniform1}
		\frac{1}{n^2}\sum_{k=1}^{n}Var(\sigma(x^*,V_{k}))\longrightarrow0
		\ \ as \ n\rightarrow\infty.
	\end{equation}
	Then
	$$P\Big\{d_{H}\Big(\frac{1}{n}\sum\limits_{k=1}^n V_{k},\frac{1}{n}\sum\limits_{k=1}^n E[V_{k}]\Big)>\varepsilon\Big\}\longrightarrow0.$$
\end{theorem}
\begin{proof}
	For each $V_{i}$, by \eqref{squareintegral}, we have that $\|V_i\|_{\mathcal {P}}\in L^2(\Omega;\mathbb R)$, and for any natural number $n$,
	\begin{equation}\label{eqn:dominated}
		\begin{split}
			&d_H\Big(\sum\limits_{i=1}^{n}V_i, E[\sum\limits_{i=1}^{n}V_i]\Big)\leq \|\sum\limits_{i=1}^{n}V_i\|_{\mathcal {P}}+E\Big[\|\sum\limits_{i=1}^{n}V_i\|_{\mathcal {P}}\Big]\\
			&\leq \sum\limits_{i=1}^{n}\|V_i\|_{\mathcal {P}}+\sum\limits_{i=1}^{n}E\Big[\|V_i\|_{\mathcal {P}}\Big]\in L^2(\Omega;\mathbb R).
		\end{split}
	\end{equation}
	
	Since $d_H\Big(\sum\limits_{i=1}^{n}V_i, E[\sum\limits_{i=1}^{n}V_i]\Big)\in L^2(\Omega; \mathbb R)$, then for any $\varepsilon>0$, by using Markov inequality and the equality \eqref{metric}, we obtain
	\begin{eqnarray}\label{eqn:3}
		&&P\Big\{d_{H}\Big(\frac{1}{n}\sum\limits_{k=1}^n V_{k},\frac{1}{n}\sum\limits_{k=1}^n
		E[V_{k}]\Big)>\varepsilon\Big\}\nonumber\\
		&&\leq\frac{1}{(\varepsilon n)^{2}}E\Big[d_{H}\Big(\sum\limits_{k=1}^n V_{k},\sum\limits_{k=1}^n E[V_{k}]\Big)\Big]^{2}\nonumber\\
		&&=\frac{1}{(\varepsilon n)^{2}}E\Big[\sup_{x^{*}\in S^{*}}|\sigma(x^{*},\sum\limits_{k=1}^n
		V_{k})-\sigma(x^{*},\sum\limits_{k=1}^n E[V_{k}])|^{2}\Big].
	\end{eqnarray}
	From the above inequality and Lemma \ref{expectsupinequality}, it holds that
	\begin{equation}
		\begin{split}
			&P\Big\{d_{H}\Big(\frac{1}{n}\sum\limits_{k=1}^n V_{k},\frac{1}{n}\sum\limits_{k=1}^n
			E[V_{k}]\Big)>\varepsilon\Big\}\\
			&\leq \frac{1}{\varepsilon^2}\sum\limits_{i=1}^{m}\frac{\sum\limits_{k=1}^n Var(\sigma(y_{i}^*,V_k))}{n^2}
			+4\sqrt{\sum\limits_{i=1}^{m}\frac{\sum\limits_{k=1}^n Var(\sigma(y_{i}^*,V_k))}{n^2}}+4\varepsilon^2. 
		\end{split}
	\end{equation}
	From the condition \eqref{eqn:uniform1}, for any $\epsilon>0$ and sufficient large $n$, we have
	$$\frac{1}{n^2}\sum\limits_{k=1}^n Var(\sigma(y_{i}^*,V_k))< \frac{\varepsilon^3}{m}.$$
	Then
	\begin{eqnarray*}
		&&\frac{1}{\varepsilon^2n^2}E\Big[\sup\limits_{x^*\in
			S^*}\Big|\sigma(x_i^*,\sum\limits_{k=1}^nV_k)-\sigma(x_i^*,\sum\limits_{k=1}^nE[V_k])\Big|^2\Big]\\
		&&\leq\frac{m}{\varepsilon^2}\frac{\varepsilon^3}{m}+4\sqrt{\frac{m\varepsilon^3}{m}}+4\varepsilon^2
		=\varepsilon+ 4\sqrt{\varepsilon^3}+4\varepsilon^2.
	\end{eqnarray*}
	By the arbitrariness of $\varepsilon$,
	the theorem is proved.
	
\end{proof}
\begin{corollary}
	Assume the dimension $dim\frak X<\infty$.  Let $\{V_{n}:n\in \mathbb N\}$ be a sequence of ${\mathcal P}_{kc}(\frak X)$-valued random variables, which are  uncorrelated, compactly uniformly integrable in $L^2$, and for any $x^*\in \frak X^*$, $\sigma(x^*, V_1), \sigma(x^*, V_2),\cdots,  $ are identically distributed. Then for arbitrary $\varepsilon >0$, we have
	$$P\Big\{d_{H}\Big(\frac{1}{n}\sum\limits_{k=1}^n V_{k},\frac{1}{n}\sum\limits_{k=1}^n E[V_{k}]\Big)>\varepsilon\Big\}\longrightarrow0.$$
\end{corollary}
\begin{proof}
	It is a special case of Theorem \ref{thm:result2}.
	In fact, for any $x^*\in S^*$,
	$$\frac{1}{n}\sum\limits_{k=1}^n Var\left(\sigma(x^*,V_k)\right)=Var\left(\sigma(x^*,V_1)\right)\leq Var(\sup_{x^*\in S^*}\sigma(x^*,V_1))=Var(\|V_1\|_{\mathcal {P}})<\infty,$$
	which implies the condition \eqref{eqn:uniform1}, that is
	$$\frac{1}{n^2}\sum\limits_{k=1}^n Var\left(\sigma(x^*,V_k)\right)=\frac{1}{n}Var(\sigma(x^*,V_1))
	\leq \frac{1}{n}Var\left(\|V_1\|_{\mathcal {P}}\right)\rightarrow0 \ \ as\ \ n\rightarrow\infty. $$
\end{proof} 
%%%%%%%%%%%%%%%%%%%%%%%%%%%%%%%%%%%%%%%%%%%%%%%%%%%%%%%%%%%%%%%%%%%%%%%%%%%%%%%%%%%%
\section {Strong law of large numbers}
\label{author_sec:4}
In this section, we shall focus on the strong law of large numbers for uncorrelated set-valued random variables when the underlying space is of finite dimension. 

First, let's list here several strong laws of large number for real-valued uncorrelated random variables that will be needed later.

\begin{theorem}\label{thmchung}
	(cf. Theorem 5.1.2 in \cite{Chung}) If $\{X_1, X_2,\cdots,\}$ is a sequence of uncorrelated real-valued random variables and their second moments have a common bound, then
	$$
	\frac{1}{n}\sum_{k=1}^{n}(X_k-E(X_k))\rightarrow  \ a.s.
	$$
\end{theorem}
\begin{theorem}\label{thmTaylor1} 
	(cf. Theorem 3.1.2 in \cite{Taylor78})	If $\{X_1, X_2,\cdots,\}$ is a sequence of uncorrelated real-valued random variables such that for all $n\in \mathbb{N}$, $Var(X_n)\leq M$ where $M$ is a constant, then
	$$
	\frac{1}{n}\sum_{k=1}^{n}(X_k-E(X_k))\rightarrow  \ a.s.
	$$
\end{theorem}

\begin{theorem}\label{thmTaylor2} 
	(cf. Theorem 3.1.3 in \cite{Taylor78})	If $\{X_1, X_2,\cdots,\}$ is a sequence of uncorrelated real-valued random variables such that 
	$$
	\sum_{n=1}^{\infty}\frac{Var(X_n)}{n^2}log^{2}n<\infty,
	$$
	then
	$$
	\frac{1}{n}\sum_{k=1}^{n}(X_k-E(X_k))\rightarrow  \ a.s.
	$$
\end{theorem}
The support function is a bridge linking set-valued variables to real-valued variables. Here, we further leverage this bridge to prove the strong law of large numbers for set-valued random variables. Let $\dim{\frak{X}}=d (>1)$, $\{V_{n},n=1,2,\cdots\}\subset \mathcal{M}(\Omega;
\mathcal{P}_{kc}(\frak{X}))$ and $x^*\in \frak{X}^*(=\frak{X})$. F We now examine the relationship between strong convergence (convergence with respect to the Hausdorff metric) and weak convergence ( convergence of the sequence of real-valued variables $\{\sigma(x^*, V_{n}), n=1,2,\cdots\}$).

\begin{lemma}\label{equivalentconvergence}
	Let $\{V, V_{n},n=1,2\cdots,\}\subset \mathcal{M}(\Omega, \mathcal{P}_{kc}(\frak{X}))$ and $x^*\in \frak{X}^*$. Then 
	\begin{equation}\label{eqstrong}
		\lim_{n\rightarrow\infty}d_{H}(V_{n}, V)=0 \ a.s.
	\end{equation}
	if and only if for any $x^*\in {\frak X}^*$,
	\begin{equation}\label{eqweak}
		\lim_{n\rightarrow\infty}|\sigma(x^*,V_{n})-\sigma(x^*,V)|=0 \ \ a.s.
	\end{equation}
\end{lemma}
\begin{proof}
	By \eqref{metric}, we know 
	\begin{equation}\label{eqmetric}
		d_{H}(V_{n}, V)=\sup\limits_{x^*\in S^*}|\sigma(x^*,V_{n})-\sigma({x^*,V})|.
	\end{equation}
	From the above equality, 
	\eqref{eqstrong}$\Rightarrow$ \eqref{eqweak} is obvious.
	It only remains to prove that  \eqref{eqweak}$\Rightarrow$ \eqref{eqstrong}.
	
	Let $\{e_1, e_2,\cdots, e_d\}$ be an orthonormal basis of $\frak X^*$. Assume that Eq.\eqref{eqweak} holds. Then for each $e_i$, $i=1,\cdots,d$,
	\begin{equation}\label{eqlimit}
		\lim_{n\rightarrow \infty}\sigma(e_i, V_n)=\sigma(e_i, V) \ a.s \ and
		\lim_{n\rightarrow \infty}\sigma(-e_i, V_n)=\sigma(-e_i, V) \ a.s.
	\end{equation}
	For any $x^*\in S^*$, there exists a set of constants $a_1,\cdots,a_d$ with $|a_i|\leq 1  (i=1,\cdots,d)$ such  that $x^*=\sum_{i=1}^{d}a_{i}e_{i}$.
	\begin{equation}\label{eqnbounded}
		\begin{split}
			\|V_n\|_{\mathcal P}&=\sup_{x^*\in S^*}\sigma (x^*,V_n)=\sup_{a_1,\cdots,a_d}\sigma (\sum_{i=1}^{d}a_{i}e_{i},V_n)\\
			&\leq \sup_{a_1,\cdots,a_d} \sum_{i=1}^{d}\sigma (a_{i}e_{i},V_n)\leq \sup_{a_1,\cdots,a_d}\sum_{i=1}^{d}|a_{i}|\left[\sigma (e_{i},V_n)+\sigma (-e_{i},V_n)\right]\\
			&\leq \sum_{i=1}^{d}\left[\sigma (e_{i},V_n)+\sigma (-e_{i},V_n)\right] \ (\ since \ each\ |a_i|\leq 1\ ).
		\end{split}
	\end{equation}
	A convergent sequence must be bounded.   Owing to the finite cardinality of the basis,  Eq.\eqref{eqlimit} and Eq. \eqref{eqnbounded},  there exists a $\mu$-null set denoted by ${\mathcal N}_{1}$ and $0<M(\omega)<\infty$ such that for $\omega\in\Omega\setminus {\mathcal N}_{1}$, $x^*\in S^*$ and all $n$, it holds that
	$$
	\sigma(x^*,V_n)\leq M(\omega).
	$$
	Therefore, we have
	\begin{equation}\label{eqbounded}
		\sup_{n\geq 1}\|V_n\|_{\mathcal P}=\sup_{n\geq 1}\sup_{x^*\in S^*}\sigma (x^*,V_n)\leq M(\omega),
	\end{equation}
	which means the sequence $\{\|V_n\|_{\mathcal P}, n=1,2,\cdots\}$ is uniformly bounded $\mu$-a.s.
	Take any $x^*, y^*\in S^*$, for any $\omega\in \Omega\setminus {\mathcal N}_{1}$, by Proposition \ref{prop} and Eq.\eqref{eqbounded}, we obtain the equi-continuity of real-valued functions $\{\sigma(\cdot, V_n),n=1,\cdots\}$ in $S^*$ since
	$$
	|\sigma(x^*,V_n)-\sigma(y^*,V_n)|\leq \|x^*-y^*\|_{\frak X^*}\|V_n\|_{\mathcal P}|\leq \|x^*-y^*\|_{\frak X^*} M(\omega).
	$$
	Applying the Ascoli-Arzel\`a  Theorem (Ref. Page 85,\cite{Yoshida}) on finite dimensional space and random case, then \eqref{eqstrong} holds.
	
	In fact, since the compactness of $S^*$, for any $\varepsilon>0$, there exists a finite cover $\{N(x^*_{1},\varepsilon),\cdots, N(x^*_{k},\varepsilon)\}\supset S^*$, where  each $x^*_{i}\in S^*$.   For any $x^*\in S^*$, there is some $x^*_{i}$, such that $\|x^*-x^*_{i}\|_{\frak X^*}\leq \varepsilon$. Then
	\begin{equation*}
		\begin{split}
			&|\sigma(x^*,V_n)-\sigma(x^*,V)|\\
			&\leq|\sigma(x^*,V_n)-\sigma(x^*_i,V_n)|+|\sigma(x^*_{i},V_n)-\sigma(x^*_i,V)|+|\sigma(x^*_i,V)-\sigma(x^*,V)|\\
			&\leq \|x^*-x^*_i\|_{\frak X^*}\|V_{n}\|_{\mathcal P}+\|x^*-x^*_i\|_{\frak X^*}\|V\|_{\mathcal P}+|\sigma(x^*_{i},V_n)-\sigma(x^*_i,V)|\\
			&\leq \varepsilon(|V_n|_{\mathcal P}+\|V\|_{\mathcal P})+|\sigma(x^*_{i},V_n)-\sigma(x^*_i,V)|.
		\end{split}
	\end{equation*}
	For the third term of the above formula, there exists a natural number $N_i$ and a $\mu$-null set ${\mathcal N}^i_2$, such that for all $n>N_i$ and $\omega\in \Omega\setminus {\mathcal N}^i_2$, 
	$$ |\sigma(x^*_{i},V_n)-\sigma(x^*_i,V)|<\varepsilon.
	$$
	Since
	$$
	\lim_{n\rightarrow \infty}|\sigma(x^*_{i},V_n)-\sigma(x^*_i,V)|=0 \ a.s. 
	$$
	Take 
	${\mathcal N}={\mathcal N}_1\cup(\cup_{i=1}^{k} {\mathcal N}_2^i)$, 
	$N=\max\{N_1,\cdots,N_k\}$. Then for  $n>N$ and any $\omega\in \Omega\setminus \mathcal N$, we have 
	$$
	\sup_{x*\in S^*}|\sigma(x^*,V_n)-\sigma(x^*,V)|\leq \varepsilon(M(\omega)+\|V\|_{\mathcal P})+\varepsilon,
	$$
	due to the arbitrariness of $\varepsilon$,  \eqref{eqstrong} holds.
	
\end{proof}
\begin{theorem} \label{thm:strong0}
	Assume $dim\frak X<\infty$ and let $\{V_{n}, n\in\mathbb N\}$ be a sequence of uncorrelated  and identically distributed $\mathcal P_{kc}(\frak X)$-valued random variables with $E[\|V_1\|_{\mathcal P}]<\infty, Var[\|V_1\|_{\mathcal P}]<\infty$.  Then
	$$d_{H}\Big(\frac{1}{n}\sum\limits_{k=1}^n V_{k},E[V_1]\Big)\longrightarrow 0\ as \ n\rightarrow \infty\ a.s.
	$$
	
\end{theorem}
\begin{proof}
	Take $x^*\in S^*$ arbitrarily, then the sequence of real-valued random variables  $\{\sigma(x^*, V_n),n=1,2,\cdots\}$ is uncorrelated and identically distributed with the same second moment. Therefore, as the special case of Theorem \ref{thmchung}, $\{\sigma(x^*, V_n),n=1,2,\cdots\}$ satisfies the strong law of large numbers. That is 
	$$|\frac{1}{n}\sum\limits_{k=1}^n \sigma(x^*,V_{k})-\sigma(x^*,E[V_1])|\longrightarrow 0\ as \ n\rightarrow \infty\ a.s.$$
	Notice that $\{E[V_1], \frac{1}{n}\sum\limits_{i=1}^{n}V_i, n=1,2,\cdots\}\subset {\mathcal P}_{kc}(\frak X)$ since $\{E[V_1],V_1,V_2,\cdots\}\subset {\mathcal P}_{kc}(\frak X)$. In Lemma \ref{equivalentconvergence}, replace $V_n$ with $\frac{1}{n}\sum\limits_{i=1}^{n}V_i$ and replace $V$ with $E[V_1]$,  the desired result is obtained immediately.
\end{proof}

If we remove the condition of identical distribution from Theorem \ref{thm:strong0}, we need to add other conditions. For example, we have the following two theorems.  

\begin{theorem} \label{thm:strong1}
	Assume $dim\frak X<\infty$ and  $\{V_{n}, n\in\mathbb N\}$ be a sequence of uncorrelated and compactly uniformly integrable (in $L^1$) $\mathcal P_{kc}(\frak X)$-valued random variables. For each n, $V_{n}\in  L^2(\Omega; {\mathcal {P}}_{kc}(\frak X))$ and   $Var(\sigma(x^*,V_{n}))\leq M$, where $M$  is a positive constant that does not depend on $x^*$ or $n$. Then
	$$d_{H}\Big(\frac{1}{n}\sum\limits_{k=1}^n V_{k},\frac{1}{n}\sum\limits_{k=1}^n E[V_k]\Big)\longrightarrow 0\ as \ n\rightarrow \infty\ a.s.
	$$
\end{theorem}
\begin{proof}
	Take $x^*\in S^*$, set $$
	S_{n}(x^*)=\sigma(x^*,\sum\limits_{i=1}^{n}V_i)-\sum\limits_{i=1}^{n}\sigma(x^*,E[V_i]).
	$$
	It remains to prove $\sup_{x^*\in S^*}\frac{|S_{n}(x^*)|}{n}\rightarrow 0$ as $n\rightarrow \infty$ a.s.
	
	Firstly, for any $x^*\in S^*$, the sequence of real-valued variables $\{\sigma(x^*,V_n),n=1,2,\cdots\}$ satisfies the conditions of Theorem \ref{thmTaylor1}. Then $\frac{|S_{n}(x^*)|}{n}\rightarrow 0$ as $n\rightarrow \infty$ a.s.
	
	Secondly, by Proposition \ref{prop}, 
	$$
	S_n(x^*)=\sum\limits_{i=1}^{n}\sigma(x^*, V_i)-\sum\limits_{i=1}^{n}\sigma(x^*, E[V_i])=\sum\limits_{i=1}^{n}\left(\sigma(x^*, V_i)-\sigma(x^*,E[V_{i}])\right).
	$$ 
	
	Since the sequence $\{V_n,n=1,2,\cdots\}$ is compactly uniformly integrable, for any $\varepsilon>0$, there exists a compact subset $\mathcal K_{\varepsilon}\subset \mathcal P(\frak X)$, such that   $$\int_{\{V_n\notin \mathcal K_{\varepsilon}\}}\|V_n\|_{\mathcal P}d\mu<\varepsilon.$$
	Since
	$$E[\|V_n\|_{\mathcal P}]=\int_{\{V_n\notin \mathcal K_{\varepsilon}\}}\|V_n\|_{\mathcal P}d\mu+\int_{\{V_n\in \mathcal K_{\varepsilon}\}}\|V_n\|_{\mathcal P}d\mu,$$ 
	then the sequence of expectations $\{E[\|V_n\|_{\mathcal P}],n=1,2,\cdots\}$ are uniformly bounded. In addition, each $E[V_n]$ is compact and convex.  By the proof of Lemma \ref{equivalentconvergence},  both the sequence of functions  $\{\sigma(\cdot, V_n),n=1,2,\cdots\}$ and the sequence of  $\{\sigma(\cdot, E[V_n]),n=1,2,\cdots\}$ are equi-continuous with respect to $x^*$ in $S^*$. From the property of equi-continuous functions, we obtain that the sequence $\{\frac{1}{n}S_{n}(\cdot),n=1,2,\cdots\}$ is also equi-continuous with respect to $x^*$ in $S^*$.
	Replace $\sigma (x^*,V_n)$ with $\sigma(x^*,\frac{1}{n}S_{n})$ and  $\sigma(x^*,V)$  with $\sigma(x^*,\{0\})$ in the proof of Lemma \ref{equivalentconvergence}, we obtain
	$\sup_{x^*\in S^*}\frac{|S_{n}(x^*)|}{n}\rightarrow 0$ as $n\rightarrow \infty$ a.s. (I.e. On the compact set $S^*$, pointwise convergence and uniform convergence of a sequence of equi-continuous functions are equivalent.)
	
\end{proof}
Furthermore, the following stronger result can be obtained.
\begin{theorem} Assume $dim\frak X<\infty$ and  $\{V_{n}, n\in\mathbb N\}$ be a sequence of uncorrelated and compactly uniformly integrable (in $L^1$) $\mathcal P_{kc}(\frak X)$-valued random variables.   And for any $x^*\in S^*$,
	\begin{equation}\label{eqn:6}
		\sum\limits_{n=1}^\infty\frac{Var(\sigma(x^*,V_n))}{n^2}\log ^2n<\infty.
	\end{equation}
	Then
	$$d_{H}\Big(\frac{1}{n}\sum\limits_{k=1}^n V_{k},\frac{1}{n}\sum\limits_{k=1}^n E[V_k]\Big)\longrightarrow 0.\ \ a.s.
	$$
\end{theorem}
\begin{proof} The proof is similar to that of Theorem \ref{thm:strong1}.
	
	Take $x^*\in S^*$, set $$
	S_{n}(x^*)=\sigma(x^*,\sum\limits_{i=1}^{n}V_i)-\sum\limits_{i=1}^{n}\sigma(x^*,E[V_i]).
	$$
	
	For each $x^*\in S^*$, by Theorem \ref{thmTaylor2}, the sequence of real-valued random variables $\{\sigma(x^*, V_n), n=1,2,\cdots\}$ obeys the strong law of large numbers. That is
	$$\frac{|S_{n}(x^*)|}{n}\rightarrow 0\  as \ \ n\rightarrow \infty\  a.s. $$
	
	The same as in Theorem \ref{thm:strong1},  the sequence $\{\frac{1}{n}S_{n}(\cdot),n=1,2,\cdots\}$ is  equi-continuous with respect to $x^*$ in $S^*$. Then
	$$\sup\limits_{x^*\in S^*}\frac{|S_{n}(x^*)|}{n}\rightarrow 0\  as \ \ n\rightarrow \infty\  a.s. 
	$$
	
\end{proof}

\begin{remark}
	
	If the probability space $(\Omega, \mathcal F, \mu)$ is non-atomic, whether $V$ is a convex set or not, its Aumman integral $E[V]$ is a convex set.  By using the Shapley and Folkman's inequality (see for example page 97 in \cite{LiOgV} or  Lemma 3.13 in \cite{Hu}), i.e.
	for $\{V_n,n\geq 1\}\subset {\mathcal P}_{k}(\frak X)$ and some $V\in \mathcal {P}_{kc}(\frak X)$, if 
	$$
	d_{H}(\frac{1}{n}\sum_{m=1}^{n}\overline{conv}V_{m}, V)\rightarrow 0 \ as\  n\rightarrow\infty,
	$$
	then
	$$
	d_{H}(\frac{1}{n}\sum_{m=1}^{n}V_{m}, V)\rightarrow 0 \ as\  n\rightarrow\infty,
	$$
	where $\overline{conv}V_{m}$ is the closed convex hull of $V_m$.
	It is easy to see that the above laws of large numbers also hold after removing the convexity condition.
	
\end{remark}
%%%%%%%%%%%%%%%%%%%%%%%%%%%%%%%%%%%%%%%%%%%%%%%%%%%%%%%%%%%%%
\section{Concluding remark}
\label{author_sec:5}
In a hyperspace, there are some known results about WLLNs and SLLNs for independent identically distributed or independent
(not necessarily identically distributed) set-valued random variables. The innovation of this paper is that the independence is replaced with a weaker condition: uncorrelation.  Since the hyperspace ${\mathcal {P}}(\frak X)$ is not linear,
there is no nice subtraction and multiplication between two sets, no suitable set-valued covariance for set-valued random variables. By means of the support function, a powerful tool that connects set-valued variables with real-valued variables. Here we firstly give the definition of uncorrelated set-valued random variables through support function and then discuss its properties. The assumption of uncorrelation is more in line with the actual situation than independence. By virtue of the nice properties of support functions in finite-dimensional spaces  we obtain the WLLN and SLLN  for uncorrelated set-valued  random variables in the sense of Hausdorff metric $d_H$.
These results are also the extension of  law of large numbers for single-valued uncorrelated random variables. Our results are expected to be used in set-valued especially interval-valued statistical modeling and analysis.

Here we considered the finite dimensional case since the unit sphere of finite dimension space is compact, which is a key condition to prove the equivalence of pointwise convergence and uniformly convergence of a sequence of functions. For infinitely dimensional case, we need stronger condition, or consider it in a weak topology. %This is also the work we are currently doing following this article.

%%%% Acknowledgments %%%%%%%%
\section*{Acknowledgment}
This work is partially supported
by Beijing Municipal Natural Science Foundation No.1192015 (Jinping Zhang)
%%%%%%%%%%%%%%%%%%%%%%%%%%%%%%%%%%%%%%%%%%%%%%%%%%%%%%%%%%%%%%%%%%%%%%
%\bibliographystyle{model1a-num-names}



\end{document}